\documentclass[11pt]{amsart}
\usepackage{amsthm}
\usepackage{array}
\usepackage{amsmath}
\usepackage{enumerate}
\usepackage{tikz}
\usetikzlibrary{calc}
\usepackage{amssymb}
\usepackage{latexsym}
\usepackage{amsfonts}
\usepackage{color}
\usepackage{mathrsfs}
\usepackage{epsfig,helvet}

\DeclareMathOperator{\Ima}{Im}

\theoremstyle{plain} \numberwithin{equation}{section}
\newtheorem{thm}{Theorem}[section]
\newtheorem{theorem}[thm]{Theorem}
\newtheorem{lemma}[thm]{Lemma}
\newtheorem{corollary}[thm]{Corollary}

\newtheorem{definition}[thm]{Definition}
\newtheorem{proposition}[thm]{Proposition}
\newtheorem*{remark}{Remark}
\begin{document}
\setcounter{page}{1}

\title[ Padhan, Nayak and Pati]{Detecting Capable Lie Superalgebra }

\author[Padhan]{Rudra Narayan Padhan}
\address{Department of Mathematics, National Institute of Technology,  \\
         Rourkela, 
          Odisha-769028 \\
                India}
\email{rudra.padhan6@gmail.com}
\author[Nayak]{Saudamini Nayak}
\address{Institute of Mathematics \& Applications,  \\
         Bhubaneswar, 751029
          Odisha\\
                India}
\email{anumama.nayak07@gmail.com}

\author[Pati]{K.C Pati}
\address{Department of Mathematics, National Institute of Technology,  \\
         Rourkela, 
          Odisha-769028 \\
                India}
\email{kcpati@nitrkl.ac.in }

\maketitle
\section{Abstract}

In this article we show that distributive law holds for non-abelian tensor product of Lie superalgebras under certain direct sums. There by we obtain a rule for non-abelian exterior square of a Lie superalgebra. We define capable Lie superalgebra and then give some chracterization. Specifically we prove that epicenter of a Lie superalgebra is equal to exterior center. Finally we classify all capable Lie superalgebras whose derived subalgebra dimension is atmost one. As an application to those results we have shown that there exists at least one non-abelian nilpotent capable Lie superalgebra $L$ of dimension $m+n \geq 3$ where $\dim L=(m \mid n)$.

\section{Introduction}

Lie superalgebras have applications in many areas of Mathematics and Theoretical Physics as they can be used to describe supersymmetry. Kac \cite{Kac1977} gives a comprehensive description of mathematical theory of Lie superalgebras, and establishes the classification of all finite dimensional simple Lie superalgebras over an algebraically closed field of characteristic zero. In the last few years the theory of Lie superalgebras has evolved remarkably, obtaining many results in representation theory and classification. Most of the results are extension of well known facts of Lie algebras. But the classification of all finite dimensional nilpotent Lie superalgebras is still an open problem like that of finite dimensional nilpotent Lie algebras.

\smallskip

In 1904 I. Schur introduced the Schur multiplier and cover of a group in his work on projective representation. Batten \cite{Batten1993} introduced and studied Schur multiplier and cover of a Lie algebra and later on, studied by several authors \cite{BS1996, Batten1996}. For a finite dimensional Lie algebra $L$ over a field $\mathbb{F}$ the free presentation of $L$ is $0\longrightarrow R \longrightarrow F\longrightarrow L$, where $F$ is a free Lie algebra and $R$ is an ideal of $F$. Then the Schur multiplier $\mathcal{M}(L)$ is isomorphic to $F' \cap R/[F,R]$. Moneyhun \cite{Moneyhun1994} proved that for a finite dimensional Lie algebra $L$ of dimension $n$, $\dim \mathcal{M}(L)=\dfrac{1}{2}n(n-1)-t(L)$, where $t(L)\geq 0$. This bound is used to classify finite dimensional nilpotent Lie algebras with some small values of $t(L)$. Specifically the complete classification of $L$ with $t(L)\leq 8$, has been depicted in \cite{Batten1996, Hardy1998, Hardy2005, Russo2011}. An improved bound for $\dim \mathcal{M}(L)$ for a non-abelian nilpotent Lie algebra $L$ is further given by Niroomand and Russo \cite{Niroomand2011} and using this some classifications of $L$ are done with lesser effort.

\smallskip

The notion of multiplier and cover for Lie algebras is generalized to the case of Lie superalgebras and studied in \cite{Nayak2018, SN2018a}. Some bounds are given for the dimension of multiplier of Lie superalgebras. Specifically the second author, using the multiplier of Lie superalgebra $L$, has defined the function $s(L)$ and classification is done for all finite dimensional nilpotent Lie superalgebra $L$ with $s(L) \leq 2$ which in turn gives classification of $L$ with $t(L)\leq 6$ \cite{Nayak2018, SN2018b}.

 \smallskip
 
 Baer \cite{Baer1938} defined the notion of capable group. Beyl et al in \cite{Beyl1979} introduced the epicenter $Z^{*}(G)$ of a group $G$ and they proved that a group $G$ is capable if and only if $Z^{*}(G)=1$. Further exterior square of a group was studied for the first time in \cite{Brown2010}, which has an interesting relation with capability of a group. Exterior center $Z^{\wedge}(G)$ of a group $G$ is defined as $Z^{\wedge}(G)=\{g \in G|~g\wedge h=1,~ \forall~h \in G\}$. Ellis \cite{Ellis1995} proved that $Z^{\wedge}(G)=Z^{*}(G)$. Similarly the notion of epicenter $Z^{*}(L)$ of a Lie algebra $L$ is given by Alamian et al \cite{Alamian2008}. The non-abelian tensor product and exterior product of Lie algebras are defined and some of the properties are studied by Ellis \cite{Ellis1987, Ellis1991, Ellis1995}. Recently, Niroomand et al \cite{PMF2013} investigated the connection between epicenter and exterior center of a finite dimensional Lie algebra. Finally, they classified all capable Heisenberg Lie algebras and in continuation, as an application they have shown that there exists atleast one capable Lie algebra of arbitrary corank.

 \smallskip
 
 The exterior product of Lie superalgebras and some of its properties are discussed in the paper \cite{GKL2015}. In this paper we prove some important properties of non-abelian tensor product as well as exterior product of Lie superalgebras. Further we define capable Lie superalgebra. Then characterization of capable Lie superalgebra are given which are extended versions of results of Lie algebra in \cite{Alamian2008}. Following \cite {PMF2013} first we classify all capable Heisenberg Lie superalgebras and then all finite dimensional capable nilpotent Lie superalgebras with dimension of derived subalgebras atmost one. We believe, this leads to the classification of non-ablian finite dimensional nilpotent Lie superalgebra $L$ with $s(L)=3$.

\section{Preliminaries}

Let $\mathbb{Z}_{2}=\{\bar{0}, \bar{1}\}$ be a field. A $\mathbb{Z}_{2}$-graded vector space $V$ is simply a direct sum of vector spaces $V_{\bar{0}}$ and $V_{\bar{1}}$, i.e., $V = V_{\bar{0}} \oplus V_{\bar{1}}$. It is also referred as a superspace. We consider all vector superspaces and superalgebras are over $\mathbb{F}$ (characteristic of $\mathbb{F} \neq 2,3$). Elements in $V_{\bar{0}}$ (resp. $V_{\bar{1}}$) are called even (resp. odd) elements. Non-zero elements of $V_{\bar{0}} \cup V_{\bar{1}}$ are called homogeneous elements. For a homogeneous element $v \in V_{\sigma}$, with $\sigma \in \mathbb{Z}_{2}$ we set $|v| = \sigma$ is the degree of $v$. A vector subspace $U$ of $V$ is called $\mathbb{Z}_2$-graded vector subspace(or superspace) if $U= (V_{\bar{0}} \cap U) \oplus (V_{\bar{1}} \cap U)$. We adopt the convention that whenever the degree function appeared in a formula, the corresponding elements are supposed to be homogeneous. 
A  Lie superalgebra (see \cite{Kac1977, Musson2012}) is a superspace $L = L_{\bar{0}} \oplus L_{\bar{1}}$ with a bilinear mapping
$ [., .] : L \times L \rightarrow L$ satisfying the following identities:

\begin{enumerate}
\item $[L_{\alpha}, L_{\beta}] \subset L_{\alpha+\beta}$, for $\alpha, \beta \in \mathbb{Z}_{2}$ ($\mathbb{Z}_{2}$-grading),
\item $[x, y] = -(-1)^{|x||y|} [y, x]$ (graded skew-symmetry),
\item $(-1)^{|x||z|} [x,[y, z]] + (-1)^{ |y| |x|} [y, [z, x]] + (-1)^{|z| |y|}[z,[ x, y]] = 0$ (graded Jacobi identity),
\end{enumerate}
for all $x, y, z \in L$. 

\smallskip

For a Lie superalgebra $L = L_{\bar{0}} \oplus L_{\bar{1}}$, the even part $L_{\bar{0}}$ is a Lie algebra and $L_{\bar{1}}$ is a $L_{\bar{0}}$-module. If $L_{\bar{1}} = 0$, then $L$ is just Lie algebra. But in general a Lie superalgebra is not a Lie algebra. Lie superalgebra without even part, i.e., $L_{\bar{0}} = 0$, is an abelian Lie superalgebra, as $[x, y] = 0$ for all $x, y \in L$. A sub superalgebra of $L$ is a $\mathbb{Z}_{2}$-graded vector subspace which is closed under bracket operation. Take $[L, L]$, it is an graded subalgebra of $L$ and is denoted as $L'$. A $\mathbb{Z}_{2}$-graded subspace $I$ is a graded ideal of $L$ if $[I,L]\subseteq I$. The ideal 
\[Z(L) = \{z\in L : [z, x] = 0\;\mbox{for all}\;x\in L\}\] 
is a graded ideal and it is called the {\it center} of $L$. If $I$ is an ideal of $L$, the quotient Lie superalgebra $L/I$ inherits a canonical Lie superalgebra structure such that the natural projection map becomes a homomorphism. By a homomorphism between superspaces $f: V \rightarrow W $ of degree $|f|\in \mathbb{Z}_{2}$, we mean a linear map satisfying $f(V_{\alpha})\subseteq W_{\alpha+|f|}$ for $\alpha \in \mathbb{Z}_{2}$. In particular, if $|f| = \bar{0}$, then the homomorphism $f$ is called homogeneous linear map of even degree. A Lie superalgebra homomorphism $f: L \rightarrow M$ is a  homogeneous linear map of even degree such that $f[x,y] = [f(x), f(y)]$ holds for all $x, y \in L$. The notions of {\it epimorphisms, isomorphisms} and {\it auotomorphisms} have the obvious meaning. Throughout, for superdimension of Lie superalgebra $L$ we simply write $\dim(L)=(m\mid n)$, where $\dim L_{\bar{0}} = m$ and $\dim L_{\bar{1}} = n$. Also throughout $A(m \mid n)$ denotes an abelian Lie superalgebra where $\dim A=(m\mid n)$. By Heisenberg Lie superalgebra we mean special Heisenberg Lie superalgebra in this article.

\smallskip

 Here we give some results on multipliers of Lie superalgebras. The concept of stem extension, cover and multiplier of Lie superalgebra is defined and studied in \cite{Nayak2018, SN2018a}. Free presentation of a Lie superalgebra $L$ is the extension $0\longrightarrow R \longrightarrow F\longrightarrow L$, where $F$ is a free Lie superalgebra, then $\mathcal{M}( L) \cong F'\cap R / [F,R]$. In \cite{GKL2015}, on can find the link between the multiplier and the second homology of of a Lie superalgebra $L$.
 
 \begin{lemma} \label{lem3.1}\cite[See Corollary 6.5]{GKL2015}
 For a free presentation $0\longrightarrow R \longrightarrow F\longrightarrow L$ of a Lie superalgebra $L$, there is an isomorphism of supermodules
 \[H_{2}(L) \cong F' \cap R / [F,R].\]
 \end{lemma}

 Now we list some  useful results from \cite{Nayak2018}, on the bound of the dimension of multiplier of a finite dimensional Lie superalgebra.
 
 \begin{theorem}\label{th3.2} \cite[See Theorem 3.3]{Nayak2018}
 Let $L=L_{\overline{0}}\oplus L_{\overline{1}}$ be a Lie superalgebra of dimension $(m \mid n)$. Then
 \[\dim \mathcal{M}(L) \leq  \dfrac{1}{2}[(m+n)^{2}+(n-m)].\]
 \end{theorem} 
  We use this Theorem \ref{th3.2} to define the $corank~t(L)$ of a nilpotent Lie superalgebra $L$ of superdimesion $(m \mid n)$ by
 \[\dim \mathcal{M}(L) =  \dfrac{1}{2}[(m+n)^{2}+(n-m)]-t(L).\] 
 
\begin{theorem}\label{th3.3}\cite[See Theorem 3.4]{Nayak2018}

  \[\dim \mathcal{M}(A(m \mid n)) = \big(\frac{1}{2}(m^2+n^2+n-m)\mid mn \big).\]
\end{theorem}

\smallskip 

A finite dimensional Lie superalgebra $L$ is said to be Heisenberg Lie superalgebra if $Z(L)=L'$ and $\dim Z(L)=1$. According to the homogeneous generator of $Z(L)$ of a Heisenberg Lie superalgebra, we can split it to even and odd Heisenberg Lie superalgebra. A special Heisenberg Lie superalgebra is a Heisenberg Lie superalgebra with even center \cite{MC2011}.

\begin{theorem} \label{th3.4}\cite[See Theorem 4.2, 4.3]{Nayak2018}
Every special Heisenberg Lie superalgebra, with even center has dimension $(2m+ \mid n)$, is isomorphic to $H(m , n)=H_{\overline{0}}\oplus H_{\overline{1}}$, where
\[H_{\overline{0}}=<x_{1},\cdots,x_{2m},z \mid [x_{i},x_{m+i}]=z,\ i=1,\cdots,m>\]
and 
\[H_{\overline{1}}=<y_{1},\cdots,y_{m}\mid [y_{j}, y_{j}]=z,\  j=1,\cdots,m>.\]
Further
$$
\dim \mathcal{M}(H(m , n))=
\begin{cases}
 (2m^{2}-m+n(n+1)/2-1 \mid 2mn)\quad \mbox{if}\;m+n\geq 2\\
(0 \mid 0) \quad \mbox{if}\;m=0, n=1\\  
   (2 \mid 0)\quad \mbox{if}\;m=1, n=0.
\end{cases}
$$

\end{theorem}
Similarly the multiplier and cover for Heisenberg Lie superalgebra of odd center is known.

\begin{theorem}\label{th3.6}\cite[See Theorem 2.8]{SN2018b}
Every Heisenberg Lie superalgebra, with odd center has dimension $(m \mid m+1)$, is isomorphic to $H_{m}=H_{\overline{0}}\oplus H_{\overline{1}}$, where
\[H_{m}=<x_{1},\cdots,x_{m} , y_{1},\cdots,y_{m},z \mid [x_{j},y_{j}]=z,   j=1,\cdots,m>.\]
Further,
$$
\dim \mathcal{M}(H_{m})=
\begin{cases}
 (m^{2}\mid m^{2}-1)\quad \mbox{if}\;m\geq 2\\  
(1\mid 1) \quad \quad \mbox{if}\;m=1.  
  \end{cases}
$$
\end{theorem}

For any two Lie superalgebras $H$ and $K$ the Lie superalgebra direct sum $H \oplus K$ is a Lie superalgebra. If $\mathcal{M}(H)$ and $\mathcal{M}(K)$ is known then $\mathcal{M}(H \oplus K)$ is given by the following result.

\begin{theorem}\label{th3.7}\cite[See Theorem 3.9]{Nayak2018}
For Lie superalgebras $H$ and $K$,
\[\mathcal{M}(H\oplus K)\cong \mathcal{M}(H)\oplus \mathcal{M}(K)\oplus (H/H'\otimes K/K').\]
\end{theorem}

\section{ Non-abelian tensor products and exterior-products of Lie superalgebras}

Recently, Garc\'{i}a-Mart\'{i}nez \cite{GKL2015} introduced the notion of non-abelian tensor product of Lie superalgebras and exterior product of Lie superalgebras. Here we recall some of the known notations and results from \cite{GKL2015}. Then we prove that distributive law holds for non-abelian tensor product of Lie superalgebras. Further it is shown that the result is true for non-abelian exterior squares.
\subsection{Non-abelian tensor product}

 Let $P$ and $M$ be two Lie superalgebras, then by an action of $P$ on $M$ we mean a $\mathbb{K}$-bilinear map of even grade, \[P\times M \longrightarrow M,~~~~~~~(p, m)\mapsto {}^ pm, \]
such that 
\begin{enumerate}
\item ${}^{[p , p']} m = {}^p({}^{p'}m)-(-1)^{|p||p'|}~  {}^{p'}({}^{p}m),$

\item ${}^p{[m , m']}=[{}^pm, m']+(-1)^{|p||m|}[m , {}^pm'],$

\end{enumerate}
for all homogeneous $p, p' \in P$ and $m, m' \in M$. For any Lie superalgebra $M$, the Lie multiplication induces an action on itself via ${}^mm'=[m , m']$.  The action of $P$ on $M$ is called trivial if ${}^pm=0$ for all $p \in P$ and $m \in M$. \\
 Given two Lie superalgebras $M$ and $P$ with action of $P$ on $M$, we define the semidirect product $M \rtimes P$ with underlying supermodule $M \oplus P$ endowed with the bracket given by
$[(m, p), (m', p')]=([m, m']+ {}^pm'-(-1)^{|m||p'|}({}^{p'}m), [p, p'])$.\\
A crossed module of Lie superalgebras is a homomorphism of Lie superalgebras $\partial : M\longrightarrow P$ with an action of $P$ on $M$ satisfying 
\begin{enumerate}
\item $\partial ({}^p{m})=[p,\partial(m)],$

\item ${}^{\partial(m)}{m'}=[m, m'],$
\end{enumerate} 
for all $p \in P$ and $m, m' \in M$.

Let $M$ and $N$ be two Lie superalgebras with actions on each other. Let $X_{M , N}$ be the $\mathbb{Z}_{2}$-graded set of all symbols $m\otimes n$, where $m \in M_{\overline{0}}\cup  M_{\overline{1}}$, $n \in N_{\overline{0}}\cup  N_{\overline{1}}$ and the $\mathbb{Z}_{2}$-gradation is given by $|m\otimes n|=|m|+|n|$. The non-abelian tensor product of $M$ and $N$, denoted by $M \otimes N$, as the Lie superalgebra generated by $X_{M , N}$ and subject to the relations:
\begin{enumerate} 
\item $\lambda (m \otimes n)=\lambda m \otimes n= m \otimes \lambda n$,
\item $(m + m')\otimes n= m\otimes n +m'\otimes n$,  where $m , m'$ have the same grade,\\
      $m \otimes (n + n')=m \otimes n + m \otimes n'$,  where $n , n'$ have the same grade,
      
\item $[m , m']\otimes n= (m\otimes {}^ {m'}{n}) -(-1)^{|m||m'|}(m'\otimes {}^{m}{n})$,\\
      $m\otimes [n,n']=(-1)^{|n'|(|m|+|n|)}({}^{n'}{m}\otimes n)-(-1)^{|m||n|}({}^{n}{m}\otimes n')$,

\item $[m\otimes n,m'\otimes n']=-(-1)^{|m||n|}({}^{n}{m} \otimes {}^{m'}{n'}),$      
\end{enumerate}
for every $\lambda \in \mathbb{K}, m, m' \in M_{\overline{0}}\cup  M_{\overline{1}}$ and $n , n' \in N_{\overline{0}}\cup  N_{\overline{1}}$. If $M=M_{\overline{0}}$ and $N=N_{\overline{0}}$ then $M \otimes N$ is the non-abelian tensor product of Lie algebras introduced and studied \cite{Ellis1991}.

 Actions of Lie superalgebras $M$ and $N$ on each other are said to be compatible if 
\begin{enumerate}
\item ${}^{({}^{n}{m})}{n'}=-(-1)^{|m||n|}[{}^{m}{n},n']$,
\item ${}^{({}^{m}{n})}{m'}=-(-1)^{|m||n|}[{}^{n}{m},m']$,
\end{enumerate}
for all $m, m' \in M_{\overline{0}}\cup  M_{\overline{1}}$ and $n, n' \in N_{\overline{0}}\cup  N_{\overline{1}}$.
For instance if $M$, $N$ are two graded ideals of a Lie superalgebra then the actions induced by the bracket are compatible.

\smallskip

 Suppose $M$ and $N$ are Lie superalgebras acting compatibly with each other. There are two Lie superalgebra homomorphisms $\mu:M \otimes N \longrightarrow M$ and $ \nu: M \otimes N \longrightarrow N$ \cite{GKL2015}. As a result of this $[M, N]^{N}$  or $[N, M]^{N}$ is the submodule generated by ${}^ mn$ and by the compatibility condition is a graded ideal of $N$. Further with some given actions of both $M$ and $N$ on $M \otimes N$,  the homomorphisms $\mu, \nu$ are crossed modules.
 
\smallskip

We consider all actions from now on are compatible and we have the following well known results \cite{GKL2015}.
\begin{proposition}\label{prop1}\cite[Proposition 3.5]{GKL2015}
Let $M$ and $N$ be Lie superalgebras acting on each other. Then the canonical map $M\otimes _{\mathbb{K}}N \rightarrow M\otimes N,\; m\otimes n \mapsto m \otimes n$, is an even, surjective homomorphism of supermodules.
\end{proposition}

Further the result below tells us when the surjective homomorphism in Proposition \ref{prop1} is an isomorphism.

\begin{proposition}\label{prop2}\cite[Proposition 3.5]{GKL2015}
 If the Lie superalgebras $M$ and $N$ act trivially on each other, then $M\otimes N$ is an abelian Lie superalgebra and there is an isomorphism of supermodules 
 \[M\otimes N \cong M^{ab}\otimes _{\mathbb{K}}N^{ab},\]
where $M^{ab}=M/[M, M]$ and $N^{ab}=N/[N, N]$. 
\end{proposition}

A Lie superalgebra tensor product is not associative but it is symmetric.

\begin{proposition}\label{prop3}\cite[Proposition 3.7]{GKL2015}
The Lie superalgebra homomorphism
$$ M \otimes N \longrightarrow N \otimes M, \quad m\otimes n \mapsto -(-1)^{|m||n|}(n \otimes m),$$ is an isomorphism.
\end{proposition}

Let us consider the category $SLie^{2}_{\mathbb{F}}$ \cite{GKL2015}. 

\begin{proposition}\label{prop4}\cite[Proposition 3.8]{GKL2015}
Given an exact sequence of Lie superalgebras
$$(0, 0) \longrightarrow (K,L) \overset{(i, j)} \longrightarrow (M, N) \overset{(\phi, \psi)} \longrightarrow (P, Q) \longrightarrow (0, 0)$$ there is an exact sequence of Lie superalgebras
$$
(K \otimes M) \rtimes (M \otimes L) \overset{\alpha} \longrightarrow M \otimes N \overset{\phi \otimes \psi} \longrightarrow P \otimes Q \longrightarrow 0.$$
\end{proposition}
Spcifically given a Lie superalgebra $M$ and a graded ideal $K$ of $M$ there is an exact sequence 
\begin{equation}\label{eq1}
(K \otimes M) \rtimes (M \otimes K) \longrightarrow M \otimes M \longrightarrow (M/K) \otimes (M/K) \longrightarrow 0.
\end{equation}

\smallskip

Below we are going to prove a lemma on non-abelian tensor product of Lie superalgebras. This leads to proof a property that distributive law holds for the non-abelian tensor product of Lie superalgebras over certain direct sums. 

\begin{lemma}\label{lemma0'}
Suppose that $L, M$ and $N$ are Lie superalgebras such that
\begin{enumerate}[(i)]
\item $L$, $N$ act compatibly on one another and $M$, $N$ act compatibly on one another;
\item ${}^l({}^mn)=(-1)^{|l||m|} ({}^m({}^ln))$, for all homogeneous elements $l \in L$ and $m \in M $, $n \in N$;
\item the canonical homomorphisms $[L , N]^{N}\otimes M \longrightarrow N \otimes M$ and $[M , N]^{N}\otimes L \longrightarrow N \otimes L$ are trivial.
\end{enumerate}
Then the Lie superalgebra direct sum $L\oplus M$ acts on $N$ by ${}^{(l , m)}n= {}^ln+{}^ mn$ and $N$ acts on $L\oplus M$ by ${}^n{(l , m)}=({}^nl, {}^nm)$ . With these actions there is an isomorphism \[(L\oplus M)\otimes N \cong (L\otimes N)\oplus (M\otimes N). \] 
\end{lemma}

\begin{proof}
Consider  $L \oplus M$ acts on $N$ as defined. Let $(l_1 , m_1)\in L_{\alpha}$ and $(l_2 , m_2)\in L_{\beta}$ for $\alpha, \beta \in \mathbb{Z}_{2}$, hence $|(l_1 , m_1)|=\alpha$ and $|(l_2, m_2)|=\beta$. Now using ${}^l({}^m n)= (-1)^{|l| |m|}({}^m({}^l n))$ we have
\begin{align*}
{}^{(l_1 , m_1)}({}^{(l_2 , m_2)}n)-(-1)^{\alpha \beta}({}^{(l_2 , m_2)}({}^{(l_1,m_1)}n))
&={}^{(l_1 , m_1)}({}^{l_2} n, {}^{m_2} n)-(-1)^{\alpha \beta}({}^{(l_2 , m_2)}({}^{l_1} n, {}^{m_1} n))\\
&={}^{l_1}({}^{l_2}n)+{}^{m_1}({}^{l_2}n)+{}^{l_1}({}^{m_2}n)+{}^{m_1}({}^{m_2}n)\\
&-(-1)^{\alpha \beta}({}^{l_2}({}^{l_1}n)+{}^{m_2}({}^{l_1}n)+{}^{l_2}({}^{m_1}n)+{}^{m_2}({}^{m_1}n))\\
&={}^{l_1}({}^{l_2}n)-(-1)^{\alpha \beta}({}^{l_2}({}^{l_1}n))+{}^{m_1}({}^{m_2}n)-(-1)^{\alpha \beta}({}^{m_2}({}^{m_1}n))\\
&= {}^{[l_1 , l_2]}n+{}^{[m_1 , m_2]}n\\
&={}^{[(l_1 , m_1), (l_2 , m_2)]}n. 
\end{align*}
Take $(l , m)\in L_{\alpha}$ and $n_1 \in L_{\beta}$ and clearly ${}^{(l , m)}[n_1 , n_2]=[{}^{(l , m)}n_1, n_2]+(-1)^{\alpha \beta}[n_1,{}^{(l , m)}n_2]$ holds.

\smallskip

 Now consider  $N$ acts on $L \oplus M$ as given. For $n_{1}\in N_{\alpha}$ and $ n_{2} \in N_{\beta}$ for $\alpha, \beta \in \mathbb{Z}_{2}$ and $l \in L$, $m \in M$ then
\begin{align*}
{}^{[n_{1},  n_{2}]}(l,m) &=({}^{[n_1 , n_2]}l, {}^{[n_1 , n_2]}m)\\
&= ({}^{n_1}({}^{n_{2}}l)-(-1)^{\alpha \beta} ({}^{n_2}({}^{n_1}l)), {}^{n_1}({}^{n_2}m)-(-1)^{\alpha \beta} ({}^{n_2} ({}^{n_1}m)))\\
&=({}^{n_1}({}^{n_2} l), {}^{n_1}({}^{n_2}m))- (-1)^{\alpha\beta}({}^{n_2}({}^{n_1} l), {}^{n_2}({}^{n_1}m))\\
&=({}^{n_1}({}^{n_2 }(l , m))- (-1)^{\alpha\beta}({}^{n_2}({}^{n_1}(l , m)).
\end{align*} 
Let $n , m \in N$ and $(l_1 , m_1), (l_2 , m_2) \in L \oplus M$ where $|n|=\alpha$ and $|(l_{1} , m_{1})|=\beta$,
\begin{align*} 
{}^n[(l_1 , m_1), (l_2 , m_2)] &= ([{}^nl_1 , l_2]+(-1)^{\alpha \beta}[l_1 , {}^nl_2],
 [{}^n m_1 , m_2]+(-1)^{\alpha\beta}[m_1 , {}^n m_2])\\
&=([{}^n l_1 , l_2], [{}^n m_1 , m_2])+(-1)^{\alpha \beta}([l_1 , {}^n l_2], [m_1 , {}^n m_2])\\
&=[({}^n l_1, {}^n m_1), (l_2 , m_2)]+(-1)^{\alpha \beta} ([l_1 , m_1], [{}^nl_2 , {}^n m_2])\\
&= [{}^n (l_1 , m_{1}), (l_2 , m_2)]+(-1)^{\alpha \beta}[(l_1 , m_1), {}^n (l_{2} , m_{2})].
\end{align*} 
Hence the defined actions are Lie actions.
Finally define 

\begin{align*}
\phi: (L \oplus M)\otimes N &\longrightarrow (L \otimes N)\oplus (M \otimes N)\\
(l,m)\otimes n\mapsto&(l\otimes n, m\otimes n).
\end{align*}
It is easy to check $\phi$ is an even linear map. Now to see $\phi$ is a Lie superalgebra homomorphism  consider for $|(l_{1},m_{1})|=\alpha,~|n_{1}|=\beta$
\begin{align*}
\phi \big([(l_{1} , m_{1})\otimes n_{1} , (l_{2} , m_{2})\otimes n_{2}] \big) &= -(-1)^{\alpha\beta}\phi \big( {}^{n_{1}}{(l_{1},m_{1})}\otimes {}^{(l_{2},m_{2})}{n_{2}}    \big) \\
& = -(-1)^{\alpha\beta} \phi \big(({}^{n_{1}}{l_{1}},{}^{n_{1}}{m_{1}})\otimes ({}^{l_{2}}{n_{2}}+{}^{m_{2}}{n_{2}}) \big)\\
&= -(-1)^{\alpha\beta} \phi \big( ({}^{n_{1}}{l_{1}},{}^{n_{1}}{m_{1}})\otimes {}^{l_{2}}{n_{2}} + ({}^{n_{1}}{l_{1}},{}^{n_{1}}{m_{1}})\otimes {}^{m_{2}}{n_{2}} \big) \\
&= -(-1)^{\alpha\beta} \big ( \big( {}^{n_{1}}{l_{1}} \otimes {}^{l_{2}}{n_{2}}, {}^{n_{1}}{m_{1}} \otimes {}^{l_{2}}{n_{2}} \big) + \big( {}^{n_{1}}{l_{1}} \otimes {}^{m_{2}}{n_{2}}, {}^{n_{1}}{m_{1}} \otimes {}^{m_{2}}{n_{2}} \big) \big ).
\end{align*}
Since ${}^{n_{1}}{m_{1}} \otimes {}^{l_{2}}{n_{2}}= {}^{n_{1}}{l_{1}} \otimes {}^{m_{2}}{n_{2}}=0 $. So we get
\begin{align*}
\phi \big([(l_{1},m_{1})\otimes n_{1}, (l_{2},m_{2})\otimes n_{2}] \big) &=-(-1)^{\alpha\beta}  \big( {}^{n_{1}}{l_{1}} \otimes {}^{l_{2}}{n_{2}},   {}^{n_{1}}{m_{1}} \otimes {}^{m_{2}}{n_{2}} \big) \\
&= \big([l_1\otimes n_1,l_2\otimes n_2],[m_1 \otimes n_1, m_2 \otimes n_2] \big)\\
&= \big[(l_1\otimes n_1,m_1 \otimes n_1),(l_2\otimes n_2,m_2 \otimes n_2) \big ] \\
&= \big[ \phi \big( (l_1 , m_1)\otimes n_1  \big),\phi \big((l_2,m_2)\otimes n_2 \big) \big].
\end{align*}
Thus $\phi$ is a Lie superalgebra homomorphism. Finally we need to find the inverse of $\phi$. Here note that the embeddings $L$ in $L\oplus M$ and $M$ in $L\oplus M$ give rise to the embeddings $\psi_{1}:L\otimes N \longrightarrow (L\oplus M)\otimes N$ and $\psi_{2}:M\otimes N \longrightarrow (L\oplus M)\otimes N$ respectively. By combining $\psi_{1}$ and $\psi_{2}$  we can get the inverse of $\phi$.
\end{proof}

As a Lie superalgebra  $P$ acts on itself by Lie multiplication, hence we can consider the tensor square $P \otimes P$.

\begin{proposition}\label{prop4a}
For Lie superalgebras $L$ and $M$,
\[(L \oplus M) \otimes (L \oplus M) \cong (L \otimes L) \oplus (L\otimes M) \oplus (M \otimes L) \oplus (M \otimes M).\]
\end{proposition}
\begin{proof}
 Consider the Lie superalgebras $L$, $M$ and $P=L \oplus M$. Let us define the actions of $L$ on $P$ (respectively $P$ on $L$) by Lie brackets. Also $M$ and $P$ act with each other by Lie brackets. But $L$ and $M$ acts trivially with each other and also $L$ and $M$ are graded ideals of $P$. So all actions are compatible and it can be readily checked that rest conditions of Lemma \ref{lemma0'} are also satisfied. So with the action of $L \oplus M$ on $P$ given as in Lemma \ref{lemma0'} and we have the isomorphism  $(L \oplus M) \otimes P= (L \otimes P) \oplus (M \otimes P)$. Further replacing with $P=L \oplus M$ we have
 \begin{align*}
 (L\oplus M)\otimes (L\oplus M) &\cong (L \otimes (L \oplus M)) \oplus (M \otimes (L \oplus M))\\
 &\cong(L\otimes L) \oplus (L\otimes M)\oplus (M\otimes L) \oplus (M \otimes M).
 \end{align*}
\end{proof}
\subsection{Non-abelian Exterior Product }
Let $P$ be a Lie superalgebra and $(M,\partial)$ and $(N,\partial ')$ be two crossed $P$-modules with the actions of $M$ and $N$ on each other via $P$. Then 

\begin{lemma}\label{lem4b}\cite[Lemma 6.1]{GKL2015}
Let $M \square N$ be the submodule of $M \otimes N$ generated by elements
\begin{enumerate}
\item $m\otimes n + (-1)^{|m'||n'|}m'\otimes n'$, where $\partial (m)=\partial '(n')$ and $\partial (m')=\partial '(n)$
\item $m_{0} \otimes n_{0}$, where $\partial (m_{0})=\partial '(n_{0})$,
\end{enumerate} 
with $m,m' \in M_{\overline{0}} \cup M_{\overline{1}},~n,n' \in N_{\overline{0}} \cup N_{\overline{1}},~m_{0} \in M_{\overline{0}}$ and $n_{0} \in N_{\overline{0}}$. Then $M \square N$ is a graded central ideal of $M \otimes N$.
\end{lemma}

\begin{definition}
Let $P$ be a Lie superalgebra and $(M , \partial)$ and $(N , \partial^{'})$ two crossed $P$-modules. Then the exterior product of $M$ and $N$ is denoted as $M \wedge N$ and is defined as \[M\wedge N= \frac{M\otimes N}{ M\square N}.\] 
\end{definition}

Most of the results on non-abelian tensor product are fulfilled by the non-abelian exterior product. Precisely there are homomorphisms of Lie superalgebras $M \wedge N \longrightarrow M$, $M \wedge N \longrightarrow N$ and actions of $M$ and $N$ on $M \wedge N$. The exterior product is symmetric. 

\smallskip

Consider a Lie superalgebra $P=P_{\bar{0}} \oplus P_{\bar{1}}$ and the identity map on $P$, i.e the map $id: P \longrightarrow P$. The map $id$, is a crossed module and hence one can think of exterior square $P \wedge P$. In $P \wedge P$, for any $p, q \in P_{\bar{0}} \cup P_{\bar{1}}$ and $p_{\bar{0}} \in P_{\bar{0}}$ we have $p \wedge q= -(-1)^{|p||q|} q \wedge p$ and $p_{\bar{0}} \wedge p_{\bar{0}}=0$.

 \begin{theorem} \label{th3} 
For Lie superalgebras $H$ and $K$, \[(H\oplus K)\wedge (H\oplus K) \cong (H\wedge H) \oplus (K\wedge K)\oplus (H/H'\otimes K/K') .\]
 Moreover
\[ Z^{\wedge}(H\oplus K)\subseteq Z^{\wedge}(H)\oplus Z^{\wedge}(K). \]
\end{theorem}
\begin{proof}
Let $(H,\partial)$ and $(K,\partial ')$ be two crossed $P=H\oplus K$-modules with the actions of $P$ on $H$ and $K$ are defined by Lie multiplication. Here $\partial$ and $\partial'$ are defined by $\partial(h) = (h, 0), \partial(k) = (0, k)$ for $h \in H$ and $k\in K$.
\begin{align*}
 (H\oplus K) \wedge (H\oplus K) &\cong (H\oplus K) \otimes (H\oplus K)+((H \oplus K) \square (H \oplus K)).\\
 &=(H \otimes H) \oplus (K\otimes K)\oplus (H \otimes K) \oplus (K \otimes H)+ H \square H+ K \square K+ H \square K+ K \square H.\\
 &=H \wedge H+ K \wedge K+ H/H'\otimes K/K'
 \end{align*}
 Since $H \square K=0= K \square H$ and with Proposition \ref{prop2} and \ref{prop3} last equality comes.
 
 \smallskip
 
  Moreover let $l+k \in Z^{\wedge}(H \oplus K)$, i.e. $(l+k) \wedge (x+y)=0$ for all $x\in H$ and for all $y \in  K$ which clearly implies $l\wedge x=0 $ for all $x \in H$ and also $k \wedge y=0$ for all $y \in K$. Hence $l+k \in Z^{\wedge}(H)\oplus Z^{\wedge}(K)$. Hence the theorem.
 \end{proof}
 Take $K=L, M=N, P=Q$ in Proposition \ref{prop4}, then the exact sequence of Lie superalgebras $0 \longrightarrow K \longrightarrow M \longrightarrow P \longrightarrow 0$ induces an exact sequence $$ K \wedge M \longrightarrow M \wedge M \longrightarrow P \wedge P \longrightarrow 0.$$
 
\section{Chracterisation of capable Lie superalgebras}\label{sec5}

We start this section by defining capable Lie superalgebra which is followed by some criterion to detect capable Lie superalgebras. 
\begin{definition}
A Lie superalgebra $L$ is said to be $capable$ if there exists a Lie superalgebra $H$ such that $L \cong H/Z(H)$. 
\end{definition}
For instance an example of capable Lie superalgebra (which is not a Lie algebra) is $H_{1}$. The Heisenberg Lie superalgebra $H_{1}$ is of dimension $(1 \mid 2)$ described by $H_{1} \cong <x, y, z \mid [x , y]=z>$. Consider the Lie superalgebra $L$ of dimension $(1\mid3)$ such that 
$L=<x, y, r, z\mid [x , y]=r, [x , r]=z>.$ Clearly $L/Z(L)\cong H_{1}$.

\smallskip

We denote $Z^{*}(L)$ to be the smallest graded ideal in $L$ such that $L/Z^{*}(L)$ is capable. 

\begin{remark}
 $Z^{*}(L)$ is a central ideal of $L$. Let $\phi$ be the isomorphism between $L/Z^{*}(L)$ and $H/Z(H)$ for some Lie superalgebra $H$. Our claim is for $x \in Z^{*}(L)$ we have $[x, y]=0$, for all $y \in L$. Here $[x, y] \in Z^{*}(L)$, so $\phi([x, y])=0=\phi(0)$ and this implies $[x, y]=0$ as desired.
 \end{remark}
 
\begin{lemma}\label{lemma0}
A Lie superalgebra $L$ is capable if and only if $Z^{*}(L)=\{0\}$.
\end{lemma}

Now we characterize capable Lie superalgebras. The results are extended versions of the results of Lie algebra in \cite{Alamian2008} to Lie superalgebra case.

 \begin{lemma}\label{lemma1}
Let $\{N_{i}\}_{i \in J}$ be a family of graded ideals of a Lie superalgebra $L$. If for each $i \in J$, the quotient Lie superalgebra $L/N_{i}$ is capable, then so is $L/ \cap_{i \in J} N_{i}$.
\end{lemma}
\begin{proof}
Since $N_{i}$ is capable for each $i \in J$, we have the following exact sequences,
\[ 0\longrightarrow Z(K_{i})\longrightarrow K_{i}\overset{\theta_i}{\longrightarrow} L/N_{i} \longrightarrow0.\]
Set $N=\cap_{i \in J} N_{i}$, is a graded ideal of $L$. Consider an arbitrary collection of Lie superalgebras $K_{i}$ with $i \in J$, for some index set $J$, then the the Cartesian product $\prod_{i \in J} K_{i}$ is a Lie superalgebra. Further set $K=\{(k_{i}) \in \prod_{i \in J} K_{i}~|~\exists ~l \in L~such ~that~\theta_{i}(k_{i})=l+N_{i},~\forall~i \in I\}$. Our claim is, $Z(K)=\prod_{i \in J} Z(K_{i})$. Let $(x_{1},x_{2},\ldots) \in \prod_{i \in J}Z(K_{i})$, i.e. $x_{i} \in Z(K_{i})$ for each ${i \in J}$. Clearly $\theta_{i}(x_{i})=N_{i}$ which implies $(x_{1},x_{2}, \ldots) \in K$ and as for all $(y_{1},y_{2}, \ldots)\in K$, we have $([x_{1}, y_{1}], [x_{2}, y_{2}], \ldots)=0$. Hence $\prod_{i \in J}Z(K_{i}) \subseteq Z(K)$. For the reverse containment let $x=(x_{1}, x_{2}, \ldots) \in Z(K)$. Consider an arbitrary element $k_{j} \in K_{j}$, then for each $i \in J$ there exists a $k_{i} \in K_{i}$ such that $\theta_{i}(k_{i})=\theta_{j}(k_{j})$. So $(k_{1}, k_{2}, \ldots )\in K$ and thus $0=[(k_{1}, k_{2}, \ldots),(x_{1}, x_{2}, \ldots)]=([k_{1}, x_{1}], \ldots)$. We get $x_{j} \in Z(K_{j})$.

Now, define a map $\psi:L/N\longrightarrow K/Z(K)$ given by $\psi(l+N)=k_{l}+Z(K)$, where $k_{l}=(k_{l,i}) \in K$ such that $\theta_{i}(k_{l,i})=l+N$ for each $i \in J$. Let $l+N, m+N \in L/N$ such that $l+N=m+N$. Then $ l-m \in N_{i}$ for all $i$ and hence $l+N_{i}=m+ N_{i}$. Then $\theta_{i}(k_{l,i}) =\theta_{i}(k_{m,i})$ and so $k_{l,i}-k_{m,i} \in \ker \theta_{i}$. Now for all $i$,  we get
$$ k_{l,i}+Z(K_{i})=k_{m,i} + Z(K_{i})$$ and hence
$\psi(l+N)=\psi(m+n)$. The map $\psi$ is well-defined. As
 $\theta_{i}$ are homogeneous maps of degree zero, so is $\psi$. Let $l+N, m+N \in L/N$ and $k_{m}=(k_{m, i}) \in K$ such that $\theta_{i}(k_{m,i})=m+Z(K_{i})$ for each $i \in J$. Then $\theta_{i}([k_{l,i},k_{m,i}])=[\theta_{i}(k_{l,i}),\theta_{i}(k_{m,i})]=[l,m]+N$, $\forall ~i \in J$. Hence $\psi[l+N, m+N]=[\psi(l+N), \psi(m+N)]$. Clearly $\psi$ is onto and finally $\ker\psi=\{l+N |k_{l, i} \in Z(K_{i}), \forall i \in J \}=N$, hence $\psi$ is one-one.
\end{proof}

Now using Lemma \ref{lemma1} we give a characterization of $Z^{*}(L)$ in terms of central extensions of the Lie superalgebra $L$.

\begin{theorem}\label{th3a}
$Z^{*}(L)=\cap\{\phi(Z(M)) \mid \phi:M\rightarrow L$ is an epimorphism with $\ker(\phi)\subseteq Z(M) \}$. 
\end{theorem}
\begin{proof}
Let $S=\cap\{\phi(Z(M)) \mid \phi:M\rightarrow L$ is an epimorphism with $\ker(\phi)\subseteq Z(M) \}$.
Since $L/Z^{*}(L)$ is capable, there exists a Lie superalgebra $M$ such that $\phi:M\rightarrow L/Z^{*}(L)$ is an epimorphism with $\ker(\phi)=Z(M)$. Consider the Lie superalgebra $L \oplus M$ and suppose that
\[K=\{(x,y)\in L\oplus M \mid \phi(y)=x+Z^{*}(L) \}.\]
Define $\theta:L\oplus M\longrightarrow L$ by $\theta(x, y)=x$. Evidently $\theta$ is an onto homomorphism of even degree. Now consider the exact sequence
\[ 0 \longrightarrow \ker \theta~\longrightarrow~K~\overset{\theta}{\longrightarrow}~L~\longrightarrow 0. \] Our claim is $\ker \theta \subseteq Z(K)$. For any $(0, y)\in \ker \theta$ where $y \in M$ and For all $(x, y')\in K$ where $x \in L$ and $y \in M$, we have
$[(0,y), (x,y')] =(0,0)$. Since $\phi(y) \in Z^{*}(L)$ implies $y \in \ker \phi=Z(M)$. Therefore the exact sequence is a central extension of $L$. Let $(x,y) \in Z(K)$, then $y \in Z(M)=\ker(\phi)$ and $0=\phi(y)=x+Z^{*}(L)$, which implies $\theta(x,y)=x \in Z^{*}(L)$. Thus $S\subseteq \theta (Z(K)) \subseteq Z^{*}(L)$.
\par
Now, for any arbitrary central extension  $0\longrightarrow A \longrightarrow B \overset{\mu}{\longrightarrow} L \longrightarrow 0$, of Lie superalgebra $L$,
 we have the following isomorphism,
\[B/A \cong L,~~\frac{B/A}{Z(B)/A} \cong B/Z(B),~~Z(B)/A \cong \mu(Z(B)).\]
Thus, $L/ \mu(Z(B))\cong B/Z(B)$, i.e $L/\mu(Z(B))$ is capable. So for each central extension of $L$ we have $L/\mu(Z(B))$ is capable hence $L/\cap\mu(Z(B))$ is capable by Lemma \ref{lemma1}. By definition of $Z^{*}(L)$ we conclude that $Z^{*}(L)=S$.
\end{proof}

\begin{proposition}\label{prop5}
For any free presentation $0\longrightarrow R \longrightarrow F \overset{\pi}{\longrightarrow} L\longrightarrow 0$ and every central extension $0\longrightarrow M \longrightarrow K \overset{\phi} \longrightarrow L\longrightarrow 0$ of $L$, we have $\overline{\pi}(Z(F/[R,F]))\subseteq \phi(Z(K))$.
\end{proposition}

\begin{proof}
Consider the central extension $0\longrightarrow \frac{R}{[R, F]} \longrightarrow \frac{F}{[R, F]} \overset{\overline{\pi}}{\longrightarrow} \frac{F}{R}\longrightarrow 0$.
By \cite[Lemma 3.2]{SN2018a}, there is a homomorphism $\beta:F/[R,F]\longrightarrow K$ and $K=M+Im~\beta$. Hence, $\beta(Z(F/[R,F]))$, i.e $\overline{\pi}(Z(F/[R,F]))\subseteq \phi(Z(K))$.
\end{proof}
An immediate consequence of Proposition \ref{prop5} is the following.

\begin{corollary}\label{cor1}
For any free presentation $0\longrightarrow R \longrightarrow F \overset{\pi}{\longrightarrow} L\longrightarrow 0$ of $L$, $Z^{*}(L)=\overline{\pi}(Z(F/[R,F]))$.
\end{corollary}

\begin{proposition}\label{prop6}
Let $L$ be a Lie superalgebra with a free presentation $0\longrightarrow R \longrightarrow F\overset{\pi}{\longrightarrow} L\longrightarrow 0$. If $S$ is an ideal in $F$ with $S/R=N$, then the following sequences are exact:
\begin{enumerate}
\item  $0\longrightarrow \frac{[S,F]\cap R}{[R,F]} \longrightarrow \mathcal{M}(L)  \overset{\sigma} \longrightarrow \mathcal{M}(L/N)\longrightarrow \frac{N\cap L^{\prime}}{[N,L]}\longrightarrow 0$.
\item  $ \mathcal{M}(L)  \overset{\sigma} \longrightarrow\mathcal{ M}(L/N)\longrightarrow \frac{N}{[N,L]}\longrightarrow \frac{L}{L^{\prime}} \longrightarrow \frac{L}{N+L^{\prime}}\longrightarrow 0$.
\item  For a central graded ideal $N$, $ N\otimes \frac{L}{L^{\prime}}\longrightarrow \mathcal{M}(L) \overset{\sigma} \longrightarrow \mathcal{M}(L/N)\longrightarrow N\cap L^{\prime} \longrightarrow 0$.
\end{enumerate}
\end{proposition}\label{prop7'}
\begin{proof}
With the given free presentation of $L$ multiplier is $\mathcal{M}(L)=\frac{R\cap F'}{[R,F]}$. Again as $L \cong F/R$, so $L/N \cong F/S$ and 
$0 \longrightarrow \frac{S}{R} \longrightarrow \frac{F}{R} \longrightarrow \frac{F}{S} \longrightarrow 0$ is free presentation of $\frac{L}{N}$. Hence $\mathcal{M}(L/N)=~\frac{S\cap F'}{[S,F]}$. Now $\frac{N \cap L'}{[N,L]} \cong \frac{S/R \cap (F/R)'}{[S/R,F/R]}=\frac{(S \cap F')+R}{[S,F]+R}$. The map $\frac{[S,F]\cap R}{[R,F]} \overset{\theta}\longrightarrow \frac{F' \cap R}{[R, F]}$ is an inclusion map and $ \Ima (\theta)=\frac{[S, F]\cap R}{[R, F]}$. Define 
$ \frac{F' \cap R}{[R, F]}  \overset{\sigma} \longrightarrow \frac{S\cap F'}{[S,F]}$ by  $x+[R, F] \mapsto x+[S, F]$. Here $\sigma$ is clearly a well defined map and a Lie superalgebra homomorphism. Now $ \ker (\sigma)= \frac{[S, F]\cap R}{[R, F]}$ and $\Ima (\sigma) =\frac{F'\cap R}{[S, F]}$. Finally consider the map $\frac{S\cap F'}{[S,F]} \overset{\theta'}\longrightarrow \frac{(S \cap F')+R}{[S,F]+R}$ is defined by $y+[S,F] \mapsto y+([S,F]+R)$. Here $\ker (\theta') =\frac{F' \cap R}{[S, F]}$ and $\Ima (\theta')= \frac{N \cap L'}{[N, L]}$. We conclude $
(1)$ is an exact sequence.

\smallskip

Clearly $\frac{N}{[N, L]} \cong \frac{S}{[S, F]+R}$, $\frac{L}{L'} \cong \frac{F}{F'+R}$ and $\frac{L}{N+L'} \cong \frac{F}{F'+S}$.
Consider the following sequence of homomorphisms $$\frac{F' \cap R}{[R, F]}  \overset{\sigma} \longrightarrow \frac{F' \cap S}{[S, F]} \overset{\psi} \longrightarrow \frac{S}{[S, F]+R} \overset{\psi'} \longrightarrow \frac{F}{F'+R} \overset{\psi''} \longrightarrow \frac{F}{F'+S} \longrightarrow 0.$$ Now $\Ima \sigma=\frac{F' \cap R}{[S, F]}= \ker \psi$, $\Ima \psi=\frac{S \cap F'}{[S,F]}= \ker \psi'$. Also $\Ima \psi'=\frac{S}{ F'+R}= \ker \psi''$ and finally $ \Ima \psi''=\frac{F}{F'+S}$ and hence we have required exact sequence $(2)$.

\smallskip

Since $N$ is central ideal,  $[S, F] \subseteq R$. Consider the bilinear map,
\begin{align*}
\frac{S}{R} \times \frac{F}{R+F'} \overset{\theta'} \longrightarrow & \frac{R\cap F'}{[R, F]}\\
(x+R, y+(R+F'))\mapsto & [x,y]+[R, F].
\end{align*}
The map $\frac{S}{R} \times \frac{F}{R+F'} \overset{f} \longrightarrow \frac{S}{R} \otimes \frac{F}{R+F'}$ is bilinear. By using universal property of tensor product there is an unique degree zero homomorphism $\theta: \frac{S}{R} \otimes \frac{F}{R+F'} \longrightarrow \mathcal{M}(L)$ such that $\theta'=\theta f$. Hence 
\begin{equation*}
\theta((x+R)\otimes y+[R, F'])= \theta'(x+R , y+[R, F'])
\end{equation*}
which implies $\Ima \theta=\frac{[S,F]}{[R, F]}$. Since $N$ is central, $\ker \sigma=\frac{[S, F]}{[R, F]}=\Ima \theta$. Further define, 
\begin{align*}
\frac{S \cap F'}{[S,F]} \overset{\theta'} \longrightarrow & \frac{(S \cap F')}{R}\\
x+[S, F] \mapsto & x+R.
\end{align*} The map is well defined as $[S, F] \subseteq R$ and $\ker \theta'=\frac{R \cap F'}{[S, F]}= \Ima \sigma$. Finally $\Ima \theta'=\frac{S \cap F'}{R}$ gives $(3)$ is exact.
\end{proof}

We now on consider $\mathcal{M}(L) \overset{\sigma} \longrightarrow \mathcal{M}(L/N)$ is the natural map.

\begin{theorem}\label{th1}
Let $N$ be a central ideal in a Lie superalgebra $L$. Then the following conditions are equivalent:
\begin{enumerate}
\item $ \frac {\mathcal{M}(L/N)}{\mathcal{M}(L)}\cong N \cap L'$
\item $N \subseteq Z^{*}(L)$.
\item The natural map $ \mathcal{M}(L)\overset{\sigma}\longrightarrow  \mathcal{M}(L/N)$ is monomorphism.
\end{enumerate}
\end{theorem}

\begin{proof}
To prove $(1)\Leftrightarrow (3)$, first consider the exact squence $(iii)$ of Proposition \ref{prop6}. If $\sigma$ is the natural monomorphism then $[S, F]=[R, F]$. Also $\mathcal{M}(L/N) \overset{\theta}\longrightarrow N \cap L'$ is an onto homomorphism and $\ker \theta=\frac{R\cap F'}{[S, F]}$. Hence $\ker \theta=\mathcal{M}(L)$. Conversely if we the isomorphism then $\frac{R \cap F'}{[S, F]}=\ker \theta=\mathcal{M}(L)$ implies we have the required monomorphism.
\smallskip

Now we want to prove $(2)\Leftrightarrow (3)$. Consider a free presentation $0 \longrightarrow R \longrightarrow F \longrightarrow L \longrightarrow 0$ of $L$ and $S$ is an ideal of $F$ with $S/R=N$. Now consider the natural map $\sigma$ and $\ker \sigma=\frac{[S,F]}{[R,F]}$. Thus, we need to check $N \subseteq Z^{*}(L)$ if and only if $Ker(\sigma)=0$, i.e $[S,F]=[R,F]$. Set, $\overline{F}=F/[R,F],~~\overline{R}=R/[R,F],~~\overline{S}=S/[R,F]$, then $[S,F]=[R,F]$ if and only if $\overline{S}\subseteq Z(\overline{F})$. By Corollary \ref{cor1}, $Z^{*}(L)=\overline{\pi}(Z(\overline{F}))$. Thus, $\overline{\pi}(\overline{S})\subseteq Z^{*}(L)$ if and only if  $\overline{S}\subseteq Z(\overline{F})$. But $\overline{\pi}(\overline{S})=N$ and the result follows. 
\end{proof}
\smallskip

Here we want to emphasize that unlike the case of Lie algebra it is not sufficient that dimension of $\mathcal{M}(L)$ less than equals to dimension of $\mathcal{M}(L/N)$ for a central ideal $N$ imply $\mathcal{M}(L)\overset{\sigma}\longrightarrow  \mathcal{M}(L/N)$ is a monomorphism. For instance let us consider the following example.
\par

 Consider the Heisenberg Lie superalgebra with odd center $H_{1} \cong <x, y , z \mid [x , y]=z>$.  We intend to give a free presentation of $H_{1}$ and there by compute multiplier of $H_{1}$. Consider a graded set $X=X_{\bar{0}} \oplus X_{\bar{1}}$, where $X_{\bar{0}}=\{x\}$ and $X_{\bar{1}}=\{y\}$. Let $F$ be the free Lie superalgebra generated by $X$
 given by\\
\[F=\big<x, y,[x, y],[y, y],[x,[x, y]],[x,[y, y]],[y,[x, y]],\ldots \big>\]
and
\[R=\big<[y, y],[x,[x, y]],[x,[y, y]],[y,[x, y]],\ldots \big>\]. Now
\[ [F,R]=\big<[x,[y,y]],\ldots \big>.\]
Clearly, $0 \longrightarrow R \longrightarrow F \longrightarrow H_{1} \longrightarrow 0$  is the free presentation of $H_{1}$.
Further
\[ F'= [F,F]=\big<[x, y],[y, y],[x,[x, y]],[x,[y, y]],[y,[x, y]],\ldots \big>,\] 

\[ F'\cap R =\big<[y, y],[x,[x, y]],[x,[y, y]],[y,[x, y]],\ldots \big>\]

\[ \mathcal{M}(H_{1})= (F'\cap R)/ [F,R] =\big<[y, y],[x,[x, y]],[y,[x, y]] \big>.\]

Using graded Jacobi identity,  $-[y, [x, y]]+[x, [y, y]]+[y,[y, x]]=0$. But  as $[x, [y, y]] \in [F, R]$, implies $[y, [x, y]]=0$. Therefore $\{[y, y],[x,[x, y]] \}$ is a basis of $\mathcal{M}(H_{1})$.

Let $N=H_{1}'=<[x, y]>$ be the central ideal of $H_{1}$. 
Consider
\[S= \big<[x, y],[y, y],[x,[x, y]],[x,[y, y]],[y,[x, y]],\ldots \big>,\]
then $S\supseteq R$ is a graded ideal of $F$ such that $N=S/R$. Hence
$0 \longrightarrow S \longrightarrow F \longrightarrow H_{1}/N \longrightarrow 0$  is the free presentation of $H_{1}/N$. So
\[ [S,F]= \big<[x,[x, y]],[x,[y, y]],[y,[x, y]],\ldots \big>,\]
$S\cap F'= S$ and
\[\mathcal{M}(H_{1}/N)=(S\cap F')/[S,F]=\big<[x, y],[y, y] \big>.\] 
So $\dim \mathcal{M}(L)=\dim \mathcal{M}(H_{1}/N)=(1 \mid 1)$. But $\mathcal{M}(H_{1}) \nsubseteq \mathcal{M}(H_{1}/N)$ and $N\cap H_{1}'=\big<[x, y] \big>$. Hence $\mathcal{M}(H_{1}) \longrightarrow \mathcal{M}(H_{1}/N)$ is not a monomorphism.

\smallskip

Epicenter of a Lie superalgebra $L$ is denoted as $Z^{\wedge}(L)$ and is defined as
\[Z^{\wedge}(L)=\{x \in L \mid x \wedge y =0, ~\forall~ y \in L \}.\]
 For a given graded ideal $N$ of $L$  we have the following. 

\begin{lemma}\label{lemma3}
Let $N$ be an graded ideal of Lie superalgebra $L$ then,
\begin{equation*}
L\wedge N\longrightarrow L\wedge L \longrightarrow L/N\wedge L/N \longrightarrow 0,
\end{equation*}
is an exact sequence.
\end{lemma}

\begin{proof}
Evidently
$ 0\longrightarrow N \longrightarrow L \longrightarrow\frac{L}{N} \longrightarrow 0$ is a short exact sequence of $L$ and hence the result follows.
\end{proof}
\begin{corollary}\label{cor2}
$ N\subseteq Z^{\wedge}(L)$ if and only if the natural map $L\wedge L \longrightarrow L/N\wedge L/N$ is a monomorphism.
\end{corollary}
\begin{proof}
Suppose
$N \subseteq Z^{\wedge}(L)$. Then $L\wedge N \overset{\phi}\longrightarrow L\wedge L \overset{\phi'}\longrightarrow L/N \wedge L/N \longrightarrow 0$ is an exact sequence and also $L \wedge N=0$ implies $\phi'$ is a monomorphism. Conversely, suppose $\phi'$ is a monomorphism implies $\Ima \phi=0$ and hence $N \subseteq Z^{\wedge}(L)$.
\end{proof}
The lemma below give a connection between the Schur multiplier and exterior square of a Lie superalgebra.

\begin{lemma}\label{lemma4}
 Let $L$ be a Lie superalgebra. Then 
\[0\longrightarrow  ~\mathcal{M}(L) \longrightarrow L\wedge L \longrightarrow L^{\prime}\longrightarrow 0 \]
is a central extension.
\end{lemma}
\begin{proof}
Let $F$ be a free Lie superalgebra on a graded set $X$.
Let \begin{equation}\label{eq2}
0 \longrightarrow  R \longrightarrow F \longrightarrow L \longrightarrow 0,
\end{equation} be a free presentation of Lie superalgebra $L$. Consider the map $\phi: \frac{F'}{[F,R]} \longrightarrow \frac{F'+R}{R}$ defined as $x+[F,R] \mapsto x+R$. Evidently $\phi$ is onto and $\ker \phi=\frac{F'\cap R}{[F,R]}$. Now $[x+[F,R], y+[F,R]]=[x,y]+[F,R]$ where $x \in F' \cap R$ and for all $y \in F'$. Here $[x,y] \in F' \cap R$ implies that $ x+[F,R] \in Z(F'/[F,R])$. Hence $$ 0\longrightarrow  \frac{F'\cap R}{[F,R]} \longrightarrow \frac{F'}{[F,R]} \longrightarrow \frac{F'+R}{R} \longrightarrow 0$$ is a central extension.

\smallskip

Now with given short exact sequence (\ref{eq2}) we have the exact sequence $R \wedge F \longrightarrow F \wedge F \longrightarrow L \wedge L \longrightarrow 0 $ and  also it is known that $F' \cong F \wedge F$ by \cite[Proposition 6.3]{GKL2015}. So $L \wedge L \cong \frac{F'}{[F,R]}$. We know $\mathcal{M}(L) \cong \frac{F'\cap R}{[F,R]}$ and as $L \cong F/R$ we have $L' \cong \frac{F'+R}{R}$. Hence    
$0\longrightarrow  ~\mathcal{M}(L) \longrightarrow L\wedge L \longrightarrow L^{\prime}\longrightarrow 0$ is a central extension.
\end{proof}
For an abelian superalgebra $A(m \mid n)$, one can easily conclude the following.
\begin{corollary}\label{cor3}
$ \mathcal{M}(A(m \mid n)) \cong A(m \mid n)\wedge A(m \mid n)$.
\end{corollary}
 
\section{Main results}
\label{sec}
Here we classify all nilpotent capable Lie superalgebras with dimension of derived subalgebra atmost one. We begin with establishing a relation between exterior center and epicenter of a Lie superalgebra.
\subsection{classification of capable nilpotent Lie superalgebra}
\begin{lemma}\label{lemma5}
For any finite dimensional Lie superalgebra $L$, we have $Z^{*}(L)=Z^{\wedge}(L).$
\end{lemma}
\begin{proof}
Consider $L \cong F/R$ for a free Lie superalgebra $F$ and a graded ideal $R$. Now $Z^{\wedge}(L)$ is a graded ideal of $L$ hence it is of the form $S/R$ for some graded ideal $S\supseteq R$ of $F$. Since, $Z^{\wedge}(L)$ is a central ideal of $L$ and $L \wedge Z^{\wedge}(L)=0$, so by Lemma \ref{lemma3} $L \wedge L \longrightarrow L/N \wedge L/N$ is an isomorphism. This along with Lemma \ref{lemma4} implies $[S,F]=[R,F]$ and hence $ \mathcal{M}(L)\longrightarrow \mathcal{M}(L/Z^{\wedge}(L))$ is a monomorphism and $Z^{*}(L)\supseteq Z^{\wedge}(L)$. To show the other containment, we show there exists a monomorphism from  $L \wedge L$ to $L/Z^{*}(L) \wedge L/Z^{*}(L)$. From the following equalities; 
\[\dim \mathcal{M}(L/Z^{*}(L))=\dim  \mathcal{M}(L)+\dim (L^{\prime} \cap Z^{*}(L)),\] 
\[\dim (L\wedge L) =\dim \mathcal{M}(L/Z^{*}(L))+\dim L^{\prime},\]
\[\dim (L/Z^{*}(L) \wedge L/Z^{*}(L))=\dim \mathcal{M}(L/Z^{*}(L))+\dim (Z^{*}(L))^{\prime}, \]
and from the isomorphism of Lie superalgebras 
\[ (L/Z^{*}(L))^{\prime} \cong (L^{\prime}+Z^{*}(L))/Z^{*}(L) \cong L^{\prime}/(Z^{*}(L) \cap L^{\prime}),\]
we have $\dim (L \wedge L) =\dim ( L/Z^{*}(L) \wedge L/Z^{*}(L))$, as desired.
\end{proof}
A Lie superalgebra is capable if and only if $Z^{\wedge}(L)=0$.

\begin{lemma}\label{lemma6}
\begin{enumerate}
\item $H(1 , 0)\wedge H(1 , 0) \cong A(3 \mid 0)$
\item $H(0 , 1)\wedge H(0 , 1)\cong A(1 \mid 0)$
\item $H(m , n)\wedge H(m , n)\cong A(r \mid s)$, where $r=(2m^{2}-m)+\dfrac{n(n+1)}{2}, s= 2mn$ and $m+n\geq 2$.
\item  $H_{1}\wedge H_{1} \cong A(1 \mid 2)$
\item  $H_{m}\wedge H_{m}\cong A(m^{2}|m^{2})$ for $m \geq 2$.
\end{enumerate}
\end{lemma}
\begin{proof}
Since $\dim H(m \mid n)^{'}=(1 \mid 0)$, so by Lemma \ref{lemma4}, $H(m , n)\wedge H(m , n)$ is an abelian Lie superalgebra where $m+n \geq 1$. Now for $m+n \geq 2$,
\[\dim H(m , n)\wedge H(m , n)=\dim H(m , n)^{\prime}+\dim \mathcal{M} (H(m , n)),\]
i.e., \[\dim H(m , n)\wedge H(m , n)=(1 \mid 0)+(2m^{2}-m+\dfrac{1}{2}n^{2}+\dfrac{1}{2}n-1 \mid 2mn).\]
\[=(2m^{2}-m+\dfrac{1}{2}n^{2}+\dfrac{1}{2}n|2mn).\]
 By setting $r=2m^{2}-m+\dfrac{1}{2}n^{2}+\dfrac{1}{2}n$ and $s=2mn$ we have the desired result for $(3)$. If $m=1, n=0$ then $\dim H(1,0) \wedge H(1, 0)=(3 \mid 0)$ and if $m=0, n=1$ then $\dim H(0,1) \wedge H(0, 1)=(1 \mid 0)$. Again $H_{m} \wedge H_{m}$ is abelian for $m \geq 1$. If $m \geq 2$ we have $\dim H_{m} \wedge H_{m}=(m^{2} \mid m^{2})$ which gives $(5)$ and if $m=1$ then $\dim H_{1} \wedge H_{1}=(1 \mid 2)$.
\end{proof}
Now we show that all finite dimensional abelian Lie superalgebras are capable except when it is an one dimensional Lie algebra.
\begin{theorem}\label{th4}
$A(m \mid n)$ is capable if and only if $m=0, n=1$ or $m+n \geq 2$.
\end{theorem}
\begin{proof}
 Consider $m=1$ and $n=0$, then $\dim \mathcal{M}(A(1 \mid 0))=0$. Clearly $\mathcal{M}(A(1 \mid 0))$ to $\mathcal{M}(A(1 \mid 0)/A(1 \mid 0))$ is a trivial monomorphism, we get $A(1 \mid 0) \subseteq Z^{\wedge}(A(1|0))$, implies $A(1|0)$ is not capable. For $m=0$ and $n=1$, $\dim (A(0 \mid 1)\wedge A(0 \mid 1))=\dim \mathcal{M}(A(0 \mid 1))=1$ and $\dim \mathcal{M}(A(0 \mid 1)/A(0 \mid 1))=0$. Thus there is no monomorphism from $A(0 \mid 1)\wedge A(0 \mid 1)$ to $A(0 \mid 1)/A(0 \mid 1)\wedge A(0 \mid 1)/A(0 \mid 1)$. Therefore $Z^{\wedge}(A(0 \mid 1))=\{0\}$ and $A(0 \mid 1)$ is capable. Let $m+n\geq 2$ and $J$ be any arbitrary ideal of $A(m \mid n)$ with $\dim J=(r \mid s)$. Thus, we have $\dim A(m \mid n)/J\wedge A(m \mid n)/J=\dim \mathcal{M}(A(m \mid n)/J)=\dfrac{1}{2}[(m+n-r-s)^{2}+(n-m+r-s)]$ and $\dim A(m \mid n)\wedge A(m \mid n)=\dim \mathcal{M}(A(m \mid n))=\dfrac{1}{2}[(m+n)^{2}+(n-m)]$. Now it follows that $A(m \mid n)\wedge A(m \mid n)$ to $A(m \mid n)/J\wedge A(m \mid n)/J$ is a monomorphism if and only if $r=s=0$, but which implies $Z^{\wedge}(A(m \mid n))=\{0\}$.
\end{proof}
Below we give a criterion on capability of Heisenberg Lie superalgebra with even center.
\begin{theorem}\label{th5}
$H(m,n)$ is capable if and only if $m=1, n=0$.
\end{theorem}
\begin{proof}
As $H(1,0)$ is Heisenberg Lie algebra, thus $H(1,0)$ is capable \cite[Theorem 3.4]{PMF2013}.
 For $m=0,n=1$, then $H(0,1)\wedge H(0,1)$ is abelian, so  $\dim H(0,1)\wedge H(0,1)=1$. Consider $J$ be the derived subalgebra of $H(0,1)$ then $H(0,1)/J \cong A(0\mid1)$ implies $\dim H(0,1)/J\wedge H(0,1)/J=1$. Thus there is a monomorphism $H(0,1)\wedge H(0,1)\longrightarrow H(0,1)/J \wedge H(0,1)/J$. Hence $J\subseteq Z^{\wedge}(H(0,1))$ and $H(0,1)$ is not capable. Assume $J=H(m,n)' \subseteq Z^{\wedge}(H(m,n))$. For $m+n\geq 2$ we have $\dim H(m,n)\wedge H(m,n)= (2m^2-m+\frac{n(n+1)}{2}\mid 2mn)$. Further $H(m,n)/J\cong A(2m \mid n)$. Hence $\dim H(m,n)/J\wedge H(m,n)/J=(2m^2+n^2+\frac{1}{2} n-m \mid 2mn)$ implies $Z^{\wedge}(H(m,n))=J$. Thus $H(m,n)$ is not capable when $m+n\geq 2$.
\end{proof}
\begin{theorem}\cite[Proposition 3.4]{SN2018b} \label{th5a}
Let $L$ be a nilpotent Lie superalgebra of dimension $(k \mid l)$ with $\dim L'=(r \mid s)$, where $r+s=1$. If $r=1, s=0$ then $L \cong H(m,n)\oplus A(k-2m-1 \mid l-n)$ for $m+n\geq 1$. If $r=0, s=1$ then $L \cong H_{m} \oplus A(k-m \mid l-m-1)$. 
\end{theorem}
So now we give criterion on capability of nilpotent Lie superalgebras of nilpotency class one.
\begin{theorem}\label{th6}
For any value of $r+s\geq1$, the nilpotent Lie superalgebra $H(m , n)\oplus A(r \mid s)$ is capable if and only if $m=1$, $n=0$.
\end{theorem}
\begin{proof}
Let us consider the following cases
\begin{enumerate}
\item $m=1,n=0;~ r=1,s=0$
\item $m=1,n=0; ~r = 0,s =1$, or $r+s \geq 2$
\item $m=0,n=1; ~r=1,s=0$
\item $m=0,n=1; ~r=0, s=1$, or $r+s \geq 2$
\item $m+n\geq2; ~r=1,s=0$
\item $m+n\geq2; ~r+s \geq 1$.
\end{enumerate}

{\bf Case-1:} Suppose $L\cong H(1 , 0)\oplus A(1 \mid 0)$. Then using Theorem \ref{th3}
\[Z^{\wedge}(L)\subseteq Z^{\wedge}\big(H(1 , 0) \big)\oplus Z^{\wedge} \big(A(1 \mid 0) \big)=A(1 \mid 0).\]
Now  $\dim Z^{\wedge}(L)$ is either $0$ or $1$. But, in one hand; 
\[\dim \mathcal{M}(L)=\dim \mathcal{M}\big(H(1,0)\big)+\dim \mathcal{M}\big(A(1 \mid 0)\big)+\dim H(1 , 0)/H(1 , 0)'\otimes A(1 \mid 0)
=2+0+2=4.\]
On the other hand, $\dim \mathcal{M}(L/A(1 \mid 0))=\dim \mathcal{M}(H(1 , 0))=2$. So, $Z^{\wedge}(L)\nsupseteq A(1 \mid 0)$, which implies $Z^{\wedge}(L)=0$.\\

{\bf Case-2:}  Suppose $L\cong H(1 , 0)\oplus A(r \mid s)$, where $r = 0, s = 1$ or $r+s \geq 2$. Then 
\[Z^{\wedge}(L)\subseteq Z^{\wedge}\big(H(1 , 0) \big)\oplus Z^{\wedge} \big(A(r \mid s) \big)=\{0\}.\]
Therefore $Z^{\wedge}(L)=\{0\}$.\\

{\bf Case-3:} If $L\cong H(0 , 1)\oplus A(1 \mid 0)$, then by a similar argument, we can have
\[\dim \mathcal{M}(L)=\dim \mathcal{M}\big(H(0 , 1)\big)+\dim \mathcal{M}\big(A(1 \mid 0)\big)+\dim~H(0 , 1)/H(0 , 1)'\otimes A(1 \mid 0)
=(0 \mid 1).\]
Denote $H(0 , 1)'=I$, so $L/I \cong H(0 , 1)/I \oplus A(1 \mid 0)$, thus $\dim\mathcal{M} (L/I)=\dim \mathcal{M}(A(0 \mid 1)\oplus A(1 \mid 0))=(1 \mid 1)$. Therefore there exist a monomorphism from $\mathcal{M}(L)$ to $\mathcal{M}(L/I)$,
Which forces to $Z^{\wedge}(L)\neq \{0\}$.\\

{\bf Case-4:} Consider $L \cong H(0 , 1)\oplus A(r\mid s)$ where $r=0, s=1$ or $r+s \geq 2$. So
\[Z^{\wedge}(L)\subseteq Z^{\wedge}\big(H(0 , 1) \big)\oplus Z^{\wedge} \big(A(r \mid s) \big)=H(0,1)'.\] Denote $H(0 , 1)'=I$, thus we have $L/I=H(0 , 1)/H(0 , 1)'\oplus A( \mid s)=A(0 \mid 1)\oplus A(r \mid s)$. Now it is easy to calculate that $\dim L\wedge L=\dim L/I\wedge L/I$. Therefore $Z^{\wedge}(L) \neq \{0\}$.

\bigskip

{\bf Case-5:} Consider $m+n\geq 2$ and $r=1,s=0$. Then $Z^{\wedge}(L)\subseteq H(m , n)'\oplus A(1 \mid 0)$. Let $I=H(m , n)'$.  Now $\dim L\wedge L=\dim L/I\wedge L/I=\dfrac{1}{2}[(2m+n)^{2}+(n-2m)]+(2m+n)$, as a result $Z^{\wedge}(L) \neq \{0\}$.

\bigskip

{\bf Case-6:} Consider $m+n \geq 2$ and $r+s \geq 1$,
\[Z^{\wedge}(L)\subseteq Z^{\wedge}\big(H(m , n) \big)\oplus Z^{\wedge} \big(A(r \mid s) \big)=H(m , n)'.\] 
\smallskip

 Clearly, $L/H(m , n)' \cong H(m , n)/H(m , n)'\oplus A(r \mid s) $. Hence by Theorem \ref{th3}; 
\[L/H(m , n)' \wedge L/H(m , n)' \cong \big(H(m , n)/H(m , n)' \wedge  H(m , n)/H(m , n)' \big) \oplus \big( A(r \mid s) \wedge A(r \mid s)\big)\]\[\oplus \big(H(m , n)/H(m , n)'\otimes A(r\mid s) \big).\]
Thus, 
\[\dim L/H(m , n)' \wedge L/H(m , n)'=\dfrac{1}{2}[(2m+n)^{2}+(n-2m)]+\dfrac{1}{2}[(r+s)^{2}+(s-r)]+(2m+n)(r+s).\]
Now,
\[\dim L \wedge~L=\dim H(m , n) \wedge H(m , n)+\dim A(r \mid s) \wedge A(r \mid s)+\dim H(m , n)/H(m , n)'\otimes A(r \mid s)\]
\[=(2m^{2}-m)+2mn+\dfrac{n(n+1)}{2}+\dfrac{1}{2}[(r+s)^{2}+(s-r)]+(2m+n)(r+s)\]
\[=\dfrac{1}{2}[(2m+n)^{2}-(2m-n)]+\dfrac{1}{2}[(r+s)^{2}+(s-r)]+(2m+n)(r+s).\]
As a result we get a monomorphism from $L\wedge L$ to $L/H(m , n)'\wedge L/H(m , n)'$, i.e $Z^{\wedge}(L)\neq \{0\}$.
\end{proof}
\smallskip

We give a criterion on capability of the odd Heisenberg Lie superalgebras.
\begin{theorem}\label{th6a}
$H_{m}$ is capable if and only if $m=1$.
\end{theorem}
\begin{proof}
Consider $m=1$, since  $H_{1}\wedge H_{1}\cong A(1 \mid 2)$ and $\dim H_{1}/I \wedge H_{1}/I=\dim A(1 \mid 1)\wedge A(1 \mid 1)=\dim \mathcal{M}(A(1 \mid 1))=(1 \mid 1)$, where $I=H'_{1}$, therefore there is no monomorphism from $H_{1}\wedge H_{1}$ to $H_{1}/I \wedge H_{1}/I$. So $Z^{\wedge}(H_{1}) = 0$ implies $H_{1}$ is capable.
\smallskip

Now let $m > 1$, then $\dim H_{m}\wedge H_{m}=(m^2 \mid m^2)$ and $\dim H_{m}/I \wedge H_{m}/I=\dim A(m \mid m)\wedge A(m \mid m)=\dim \mathcal{M}(A(m \mid m))=(m^2 \mid m^2)$, where $I=H'_{m}$. Therefore there is a monomorphism from $H_{m}\wedge H_{m}$ to $H_{m}/I \wedge H_{m}/I$. Thus $Z^{\wedge}(H_{m}) \neq \{0\}$. 
\end{proof}

\begin{theorem}\label{th7}
For any value of $r+s\geq1$, the nilpotent Lie superalgebra $H_{m}\oplus A(r \mid s)$ is capable if and only if $m=1$.
\end{theorem}
\begin{proof}
Consider the following cases;
\begin{enumerate}
\item $m=1;~r = 1, s = 0$

\item $m=1;~r = 0, s = 1$ or $r+s \geq 1$
\item $m \geq 2;~r = 1, s = 0$
\item $m \geq 2;~ r=0, s=1$ or $r + s \geq 2$.
\end{enumerate}

{\bf Case-1:} Let $L \cong H_{1}\oplus A(1 \mid 0)$. Then
\[Z^{\wedge}(L)\subseteq Z^{\wedge}(H_{1})\oplus Z^{\wedge}(A(1 \mid 0))=Z^{\wedge}(A(1 \mid 0))=A(1 \mid 0). \]
Now consider $I=A(1 \mid 0)$, then $\dim (L\wedge L)=\dim (H_{1}\oplus A(1 \mid 0) \wedge H_{1}\oplus A(1 \mid 0))=5$ and $\dim (L/I \wedge L/I)= \dim (H_{1} \wedge  H_{1})=3$. Therefore there is no monomorphism from $L\wedge L$ to $L/I \wedge L/I $. Thus $Z^{\wedge}(L)=\{0\}$.

\bigskip

{\bf Case-2:} Let $L \cong H_{1}\oplus A(r \mid s)$, where $r =0 , s = 1$ or $r+s \geq 2$. Then 
\[Z^{\wedge}(L)\subseteq Z^{\wedge}(H_{1})\oplus Z^{\wedge}(A(r \mid s))=\{0\}, \]
i.e $Z^{\wedge}(L)=\{0\}$.

\bigskip

{\bf Case-3:} Let $L \cong H_{m}\oplus A(1 \mid 0)$, where $m \geq 2 $. Then
  \[Z^{\wedge}(L)\subseteq Z^{\wedge}(H_{m})\oplus Z^{\wedge}(A(1 \mid 0))=H'_{m}\oplus A(1 \mid 0). \]
 Consider $I=H'_{m}$, then $dim (L\wedge L)= 2m^{2}+2m$ and $dim (L/I \wedge L/I)= 2m^{2}+2m$. Hence $Z^{\wedge}(L) \neq \{0\}$. 
 
\bigskip

{\bf Case-4:} Let $L \cong H_{m}\oplus A(r \mid s)$, where $m \geq 1;~r =0 , s=1$ or $r+s \geq 2$. Then 
\[Z^{\wedge}(L)\subseteq Z^{\wedge}(H_{m})\oplus Z^{\wedge}(A(r \mid s))=Z^{\wedge}(H_{m}))=H'_{m}. \]
 Now let $I=H'_{m}$, then  $\dim L\wedge L = 2m^{2} + \dfrac{1}{2}[(r+s)^{2}+(s-r)]+2m(r+s)$ and $\dim L/I \wedge L/I= \ dim \big(H_{m}/H'_{m}\oplus A(r \mid s)~\wedge ~H_{m}/H'_{m}\oplus A(r \mid s) \big)=2m^{2} + \dfrac{1}{2}[(r+s)^{2}+(s-r)]+2m(r+s)$. Thus $Z^{\wedge}(L) \neq \{0\}$.
 \end{proof}
All nilpotent Lie superalgebras with nilpotency class one are described by Theorem \ref{th5a}. Now, we give the classification of capable nilpotent Lie superalgebras whose derived subalgebra dimension is one.

\begin{theorem}\label{th8}
Let $L$ be a finite dimensional nilpotent Lie superalgebra with derived sualgebra  dimension is one. Then $L$ is capable if and only if either $L \cong H(1 , 0)\oplus A(k-3 \mid l)$ or $L \cong H_{1}\oplus A(k-1 \mid l-m-1)$.
\end{theorem}
\begin{proof}
The proof follows from Theorem \ref{th6} and Theorem \ref{th7}.
\end{proof}

\begin{theorem}\label{th9}
For a nilpotent Lie superalgebra $L$ with $\dim L'=(r \mid s)$, where $r+s=1$ and $\dim L/Z(L)>2$, is not capable.
\end{theorem}
\begin{proof}
From Theorem \ref{th5a}, if $r=1, s=0$ we have $L \cong H(m , n)\oplus A(k-2m-1 \mid l-n)$. Then 
\begin{align*}
L/Z(L)  &\cong \frac{H(m , n)\oplus A(k-2m-1  \mid l-n)}{Z(H(m , n)\oplus A(k-2m-1 \mid l-n))}\\
 &\cong \frac{H(m , n)}{Z(H(m , n))}\oplus \frac{A(k-2m-1 \mid l-n)} {A(k-2m-1 \mid l-n)}\\
 &\cong \frac{H(m , n)}{Z(H(m , n))} \cong A(2m \mid n).
\end{align*}
Since, $\dim L/Z(L)>2$, we have possible values for $m$ are $m=0$, $m=1$ or $m > 1$ which imply $n>2$, $n \geq 1$ or $n\geq 0$ respectively. Now using Theorem \ref{th6} we conclude that $L$ is not capable. If $\dim L'=(0 \mid 1)$ then using Theorem \ref{th5a}, $L\cong H_{m}\oplus A(m \mid m-1)$. Evidently we have $L/Z(L) \cong \frac{H_{m}}{Z(H_{m})}$. Again as $\dim L/Z(L) >2$ so $m >1$. Hence by Theorem \ref{th7}. Hence $L$ is not capable.
\end{proof}
\subsection{An application }

      We list here all finite dimensional nilpotent Lie superalgebras $ L$ with corank $t(L) \leq 4$. With this result in hand we show that there always exist finite dimensional nilpotent capable Lie superalgebra of each corank.

\begin{theorem}\cite [See Theorem 5.1]{SN2018b}\label{th10}
Let $L$ be a nilpotent Lie superalgebra, $\dim L=(m \mid n)$. Then
\begin{enumerate}[(i)]
\item $t(L)=0$ if and only if $L \cong A(m \mid n)$
\item $t(L)=1$ if and only if $L \cong H(1 , 0)$\\
\item $t(L)=2$ if and only if $L \cong H(1 , 0)\oplus A(1 \mid 0)$ or $H(0 , 1)$\\
\item $t(L)=3$ if and only if $L \cong H(1 , 0)\oplus A(2 \mid 0)$, $H(0 , 1)\oplus A(1 \mid 0)$, $H(0 , 2)$, $H_{1}$ or $H(0 , 1)\oplus A(0 \mid 1)$\\
\item $t(L)=4$ if and only if $L \cong H(1 , 0)\oplus A(3 \mid 0)$, $L_{5,0}$, $L_{4 , 0}$, $H(0 , 3)$, $H_{1} \oplus A(1 \mid 0)$, $H(0,1)\oplus A(1 \mid 1)$, $H(0 , 1)\oplus A(0 \mid 2)$, $H(0 , 1)\oplus A(2 \mid 0)$ or $H_{1}\oplus A(0 \mid 1)$  or $H_{1} \oplus A(2 \mid 0).$\\
\end{enumerate}
\end{theorem}

Now we obtain a result to find corank of direct sum of two finite dimensional Lie superalgebras when corank of each is known.

\begin{lemma}
Let $L$ and $K$ be two Lie superalgebras of dimension $(m \mid n)$ and $(r \mid s)$. Then
\[t(L\oplus K)=t(L)+t(K)+(m+n)(r+s)-\dim(L/L'\otimes K/K') \]
\end{lemma}
\begin{proof}
By definition of corank and using Theorem \ref{th3.7}
\begin{align*}
t(L\oplus K) &=\dfrac{1}{2}(m+r)(m+r-1)+\dfrac{1}{2}(n+s)(n+s+1)+(m+n)(r+s)-\dim M(L\oplus K)\\
&= \dfrac{1}{2}(m+r)(m+r-1)+\dfrac{1}{2}(n+s)(n+s+1)+(m+n)(r+s)-\dfrac{1}{2}m(m-1)\\
&-\dfrac{1}{2}n(n+1)-mn+t(L)-\dfrac{1}{2}r(r-1)-\dfrac{1}{2}s(s+1)-rs+t(K)-\dim(L/L'\otimes K/K')\\
&=t(L)+t(K)+(m+n)(r+s)-\dim(L/L'\otimes K/K').
\end{align*}
\end{proof}
Precisely we mentation corank of Heisenberg Lie superalgebra and abelian Lie superalgebra as follows.
\begin{lemma}
$t(H(m , n))=2m+n+1$ for $m+n\geq2$, $t(H_{m})=2m+2$ for $m\geq2$ and $t(A(m \mid n))=0$ for $m+n\geq 1$.
\end{lemma}
\begin{proof}
As a consequence of Theorem \ref{th3.7} and \ref{th3.4}, for $m+n\geq 2$,
 \[t(H(m , n))=\dfrac{1}{2}[(2m+n+1)^{2}+(n-2m-1)]-(2m^{2}-m+\dfrac{1}{2}n(n+1)+2mn-1).\] Hence $t(H(m , n))=2m +n+1$ and
similarly we conclude the rest results by Theorem \ref{th3.7}, \ref{th3.6}, \ref{th3.3}.
\end{proof}

\begin{theorem}
There exists at least one finite dimensional capable nilpotent Lie superalgebra of arbitrary corank.
\end{theorem}
\begin{proof}
For $m+n\geq2$ and $r+s\geq1$, we have 
\begin{align*}
t(H(m , n)\oplus A(r \mid s)) &= 2m+n+1+0+(2m+1+n)(r+s)-(2m+n)(r+s)\\
&=2m+n+r+s+1.
\end{align*}
Also,
 \[t(H_{p}\oplus A(r \mid s))=2p+2+(r+s).\]
Thus $H(m , n)\oplus A(r \mid s)$ and $H_{p}\oplus A(r \mid s)$ are nilpotent Lie superalgebras of corank greater than one, but from Theorem \ref{th6} and Theorem \ref{th7} we know that $H(m , n)\oplus A(r \mid s)$ and $H_{p}\oplus A(r \mid s)$ are capable if and only if $m=1, n=0$ or $p=1$ respectively. Thus there exists at least one capable Lie superalgebra of corank $\geq 4$. But from Theorem \ref{th10} with Theorem \ref{th4}, \ref{th5}, \ref{th6}, \ref{th6a} and \ref{th7} we have our desired result.
\end{proof}
\begin{corollary}
 There is a non-abelian nilpotent capable Lie superalgebra $L$, with $\dim L=(m \mid n)$, for each dimension $m+n \geq 3$.
\end{corollary}

\end{document}